\documentclass[12pt]{article}

\usepackage{amsfonts,amssymb}

\usepackage{amsthm}

\usepackage{amsmath}

\theoremstyle{plain}

\newtheorem{theorem}{Theorem}[section]
\newtheorem{lemma}[theorem]{Lemma}
\newtheorem{proposition}[theorem]{Proposition}
\newtheorem{corollary}[theorem]{Corollary}
\newtheorem{Counter-example}[theorem]{Counter-example}
\newtheorem{remark}[theorem]{Remark}

\theoremstyle{definition}

\theoremstyle{remark}

\usepackage{amssymb}
\usepackage{color}
\usepackage{a4wide}

\def\R{\mathbb R}

\def\S{\mathbb S}
\def\L{\mathbb L}

\def\E{\mathbb E}

\begin{document}

\title{A new approach to the study of spacelike submanifolds in a spherical Robertson-Walker spacetime:\\characterization of the stationary spacelike submanifolds as an application}

\date{}

\maketitle

\vspace*{-20mm}

\begin{center}
\Large D. Ferreira, E.A. Lima Jr., F.J. Palomo and A. Romero\footnote{The first author was partially supported by CAPES, Brazil. The  second  authors is partially supported by CNPq, Brazil, PQ-2 Grant 309668/2021-2.. The third author was partially
supported by Spanish MICINN project PID2020-118452GB-I00. The fourth named author was partially supported by the Spanish MICINN and ERDF project PID2020-116126GB-I00. The third and the fourth authors by the Andalusian and ERDF project A-FQM-494-UGR18. Research partially supported by the ``Mar\'{\i}a de Maeztu'' Excellence Unit IMAG, reference CEX2020-001105-M, funded by
MCIN-AEI-10.13039-501100011033.}

\end{center}

\thispagestyle{empty}

\begin{abstract}
A natural one codimension isometric embedding of each $(n+1)$-dimensional spherical Robertson-Walker (RW) spacetime $I\times_f \mathbb{S}^n$ in $(n+2)$-dimensional Lorentz-Minkowski spacetime
 $\mathbb{L}^{n+2}$ permits to contemplate $I\times_f \mathbb{S}^n$ as a rotation Lorentzian hypersurface of $\mathbb{L}^{n+2}$. After a detailed study of such Lorentzian hypersurfaces, any $k$-dimensional spacelike submanifold of such an RW spacetime can be contemplated as a spacelike submanifold of 
 $\mathbb{L}^{n+2}$. Then, we use that situation to study $k$-dimensional stationary (i.e., of zero mean curvature vector field) spacelike submanifolds of the RW spacetime. In particular, we prove a wide extension of the Lorentzian version of the classical Takahashi theorem, giving a characterization of stationary spacelike submanifolds of $I\times_f \mathbb{S}^n$ when contemplating them as spacelike submanifolds of $\mathbb{L}^{n+2}$. 
\end{abstract}

\section{Introduction}

For any isometric immersion $\Psi\colon M^{k} \to \R^{m}_{s}$  of a $k$-dimensional Riemannian manifold, $M^k$, in an $m$-dimensional semi-Euclidean space of arbitrary signature $s$, $0\leq s \leq m$, $\R^{m}_{s}$, the position vector field $\Psi$ is closely related to the extrinsic geometric of the immersion by means of the well-known Beltrami formula
\begin{equation}\label{Beltrami}
\Delta \Psi = k \mathbf{H},
\end{equation}
where $\Delta$ denotes the Laplacian operator on $M^k$ and $\mathbf{H}$ is the mean curvature vector field of $\Psi$.
This elegant and simple formula permits translate geometric assumptions on $\mathbf{H}$ into analytic ones on $\Psi$. For instance, $M^k$ is stationary, i.e., $\mathbf{H}=0$, if and only if the components of $\Psi$ are harmonic functions on $M^k$. Conversely, assumptions on $\Psi$ involving $\Delta$ are also translated to conditions on $\mathbf{H}$, for instance, Takahashi proved in 1966 that
\begin{quote}{\it 
If an isometric immersion $\Psi : M^k \to \R^m$ of a Riemannian manifold $M^k$ in Euclidean space $\R^m$ satisfies 
\begin{equation}\label{eigenvalue}
\Delta \Psi + \lambda \Psi =0,
\end{equation}
for some constant $\lambda \neq 0$, then $\lambda$ is necessarily positive, and $\Psi$ realizes a stationary immersion in an hypersphere $\S^{m-1}(\sqrt{k/\lambda})$ of radius $\sqrt{k/\lambda}$ in $\R^m$. Conversely, if $\Psi$ realizes a stationary immersion in a hypersphere  of radius $R$ in $\R^m$, then $\Psi$ satisfies (\ref{eigenvalue}) up to a parallel displacement in $\R^{m}$ and $\lambda=m/R^2$, {\rm \cite[Thm. 3]{Takahashi}}. 
}
\end{quote}

\vspace{1mm}

The extension of this result to spacelike submanifolds in Lorentz-Minkowski spacetime $\L ^m$ was obtained by Markvorsen, as a particular case of \cite[Thm. 1]{Markvorsen} and it reads as follows: 
\begin{quote}{\it 
If an isometric immersion $\Psi : M^k \to \L^m$ of a Riemannian manifold $M^k$ in Lorentz-Minkowski spacetime $\L^m$ satisfies (\ref{eigenvalue})
for some constant $\lambda >0$, then $\Psi$ realizes a stationary spacelike immersion in an $(m-1)$-dimensional De Sitter spacetime $\S_1^{m-1}(\sqrt{k/\lambda})$ of radius $\sqrt{k/\lambda}$ in $\L^m$. 
Conversely, if $\Psi$ realizes a stationary spacelike immersion in $\S_1^{m-1}(R)$, then $\Psi$ satisfies $(\ref{eigenvalue})$ up to a parallel displacement in $\L^{m}$ and $\lambda=m/R^2$
}
\end{quote}

\vspace{1mm}

An $(n+1)$-dimensional spherical Robertson-Walker (RW) spacetime $I\times_{f}\S^{n}$ is the product manifold $I\times \S^{n}$, where $I$ is an open interval of $\R$ and $\S^{n}$ the $n$-dimensional unit round sphere, endowed with the Lorentzian metric 
\begin{equation}\label{g^f}
	g^{f}=-\pi_{I}^{*}(ds^{2})+f(\pi)^{2}\pi^{*}(g),
\end{equation}
where $f>0$ is a smooth function on $I$, and $\pi_{I}$, $\pi$
denote the projections onto $I$ and $\S^{n}$, respectively, and $ds^2$ and $g$ are the usual Riemannian metrics on $I$ and $\S^{n}$, respectively. This Lorentzian manifold is a warped product, in the sense of \cite[Def. 7.33]{O'Neill} with base $(I,-ds^2)$, fibre $(\S^{n},g)$ and warping function $f$. As it is well-known,  an $(n+1)$-dimensional De Sitter spacetime $\S_1^{n+1}(R)$, of arbitrary radius $R>0$, can be seen as the spherical RW spacetime $\R\times_f\S^{n}$, where $f(t)=R\cosh (t/R)$, \cite{O'Neill}. On the other hand, the $(n+1)$-dimensional static Einstein spacetime, namely $ \R\times \S^{n}$ endowed with the Lorentzian metric (\ref{g^f}) where  $f=1$, is trivially a spherical RW spacetime.

\vspace{1mm}

Now, for a given spherical RW spacetime $I\times_{f}\S^{n}$, we assume $0\in I$ without loss of generality. Consider $h\in C^{\infty}(I)$ given by $h'(s)=\sqrt{1+f'(s)^2}>0$, for all $s\in I$ and $h(0)=0$. Take $J:=h(I)$ and $r:=f\circ h^{-1}$ that satisfies 
\begin{equation}\label{admissible}
r(t)>0 \;\; \text{and \; \,}|r'(t)|<1, 
\end{equation}
for all $t \in J$. 
Then, the map $\psi:I\times _{f}\S^{n}\to \L^{n+2}$ given by
\begin{equation}\label{isometric_embedding}
\psi(s,p)=\big(h(s), f(s)p\big),
\end{equation}
for any $(s,p)\in I\times _{f}\S^{n}$, is an isometric embedding, \cite{Akbar}, that allows us 
contemplate $I\times_f\S^{n}$ as a rotation Lorentzian hypersurface of $\L^{n+2}$. 

\vspace{1mm}

The main goal of this article is twofold. First of all, we describe carefully the geometry of the isometric embedding (\ref{isometric_embedding}), that permits think about  each spacelike submanifold of $I\times_f\S^{n}$ as a spacelike submanifold of $\L^{n+2}$. Secondly, as an application, we prove a Takahashi type result for spacelike submanifolds in an arbitrary spherical RW spacetime, that widely extends for the Lorentzian signature Markvorsen's previously quoted theorem.

\vspace{1mm}

The content of this paper is organized as follows. In Section 2, we identify $I\times_f\S^{n}$ to the rotation Lorentzian hypersurface $Q(r):=\psi(I\times_f\S^{n})\subset \L^{n+2}$, where $r$ is constructed from $f$ as above, in particular it satisfies (\ref{admissible}). Conversely, if we put 
\begin{equation}\label{hyperquadric}
Q(r)=\Big\{(t,x)\in J\times \E^{n+1}\;:\;  \sum_{i=1}^{n+1}x_{i}^{2}=\| x \|^2=r(t)^2\Big\}\subset \L^{n+2},
\end{equation}
where $r\in C^{\infty}(J)$, satisfies (\ref{admissible}) on an open interval $J\subset \R$, with $0\in J$. Then, we have $Q(r)=\psi(I\times_f\S^{n})\subset \L^{n+2}$, where the warping function $f$ is naturally obtained from $r$ reversing the previous construction of $r$ from $f$. To simplify notation, a function $r\in C^{\infty}(J)$ that satisfies (\ref{admissible}) will be called admissible in all following. On the other hand, since $Q(r)$ is isometric to $I\times_f\S^{n}$, all the intrinsic geometric properties of the spherical RW spacetime are automatically translated to $Q(r)$, in particular, we show that each Lorentzian hypersurface $Q(r)$ admits a timelike conformal and closed vector field, namely $\psi_{*}(f\,\partial /\partial t)$. This section ends with two remarks: the first one concerning the problem to find an isometric embedding of a given Lorentzian manifolds in some Lorentz-Minkowski spacetime, Remark \ref{relacion_con_MS}, 
and the second one, comparing our setting with the Riemannian case, Remark \ref{comparacion_con_caso_riemanniano}. 

\vspace{1mm}
 
Section 3 is devoted to study the extrinsic geometry of $Q(r)$. First of all, it is shown that $Q(r)$ is a rotation Lorentzian hypersurface of $\L^{n+2}$. Later, the corresponding fundamental formulae as stated, in particular, the Weingarten operator is explicitly obtained in Lemma \ref{quasiumbilical}. This result also shows that 
$Q(r)$ is actually quasiumbilical. The case totally umbilical is characterized in terms of a differential equation involving the function $r$, Remark \ref{2 de mayo}, whose solutions give rise to a De Sitter spacetime of arbitrary radius, Remark \ref{DeSitterQr}. Moreover, the fact that $Q(r)$ has proporcional principal curvatures \cite{CH} is characterized in Proposition \ref{ppc}.

\vspace{1mm}

In Section 4, we prove the announced application,   
Proposition \ref{lemma11} and Theorem \ref{010222}: 
\begin{quote}{\it 
If an isometric immersion $ \Psi\colon M^k \to \L^{n+2}$ of a Riemannian manifold $M^k$ in $\L^{n+2}$ satisfies
\begin{equation}\label{120322A}
\Delta \Psi +q_{\Psi_{0}}\,\mathbf{P}=0,
\end{equation}
where  $\mathbf{P}$ is the vector field along the immersion $\Psi=(\Psi_{0}, \Psi_{1}, \dots , \Psi_{n+1})$ given by
\begin{equation}\label{110322A}
\mathbf{P}=\big(r(\Psi_{0})r'(\Psi_{0}), \Psi_{1}, \dots , \Psi_{n+1}\big),
\end{equation}
and 
\begin{equation}\label{110322B}
q_{\Psi_{0}}=\linebreak\frac{k-\big(r''(\Psi_{0})r(\Psi_{0})+  r'(\Psi_{0})^{2}-1\big)\| \nabla \Psi_{0} \|^{2}}{r(\Psi_{0})^2(1- r'(\Psi_{0})^{2})}\in C^{\infty}(M^k)\,,
\end{equation}
for some admissible function $r$, i.e., when
the components of the spacelike immersion $\Psi\colon M^k \to \L^{n+1}$ satisfy
$$
\hspace*{-25mm}\Delta \Psi_{0}+q_{\Psi_{0}}\,r(\Psi_{0})\,r'(\Psi_{0})=0, 
$$
$$ 
\Delta \Psi_{i}+q_{\Psi_{0}}\, \Psi_{i}=0, \quad \quad i=1,2,\dots, n+1\,,
$$
then, $\Psi$ realizes a stationary spacelike immersion in $Q(r)$.
Conversely, if $\Psi$ realizes, for some admissible function $r$ with $q_{\Psi_{0}}>0$, a stationary spacelike immersion in the  Lorentzian hypersurface $Q(r)$, then, equation $(\ref{120322A})$ holds true. 
}
\end{quote}
It should be noticed that if $r(t)=\sqrt{1+t^2}$, for all $t\in \R$, $Q(r)$ is the unitary De Sitter spacetime  $ \S^{n+1}_{1}$, Remark \ref{251122}, then previous results specializes the aforementioned Markvorsen's result. 
On the other hand,
if $r(t)=1$, for all $t\in \R$, $Q(1)$ is the static Einstein spacetime  $\R\times \S^{n}$. Our result specializes now, Corollary \ref{caso_particular}  and Theorem \ref{010222}: 
\begin{quote}{\it
An isometric immersion $ \Psi\colon M^k \to \L^{n+2}$ of a Riemannian manifold $M^k$ in $\L^{n+2}$ satisfies
\begin{equation}\label{cylinder}
\Delta \Psi_{0}=0, \quad \; \Delta \Psi_{i}+(k+ \|\nabla \Psi_{0} \|^{2})\Psi_{i}=0,\; i=1,2,\dots, n+1,
\end{equation}
if and only if $\Psi$ realizes a stationary spacelike immersion in the static Einstein spacetime $Q(1)$,
}
\end{quote}
As an immediate consequence, the only compact stationary spacelike submanifolds in $Q(1)$ are the stationary submanifolds of a slice $\{t_0\}\times \S^{n}\equiv \S^{n}$.  

\vspace{1mm}

We also give several applications of Proposition \ref{lemma11} to physically realistic spherical RW spacetimes, Corollary \ref{3105B}.

\hyphenation{Lo-rent-zi-an}
\section{Preliminaries}
Let $\L^{n+2}$ be the $(n+2)$-dimensional Lorentz-Minkowski spacetime, that is, $\R^{n+2}$ endowed with the Lorentzian metric 
\begin{equation}\label{metric}
\langle\;\, ,\;\, \rangle =
-(dt)^{2}+(dx_{1})^{2}+(dx_{2})^{2}+...+(dx_{n+1})^{2}, 
\end{equation}
where
$(t,x_{1},x_{2},..., x_{n+1})=(t,x)\in \R\times \R^{n+1}$ are the usual coordinates of
$\R^{n+2}$. 

\vspace{1mm}

Given an $(n+1)$-dimensional spherical RW spacetime $I\times _{f}\S^{n}$,  let $h\in C^{\infty}(I)$ be defined by $h'(s)=\sqrt{1+f'(s)^2}>0$, for all $s\in I$ and $h(0)=0$. Take $J:=h(I)$ and $r:=f\circ h^{-1}>0$. It is clear that $r$ satisfies
\begin{equation}\label{RWEQ}
r'(t)=\frac{f'(s)}{\sqrt{1+f'(s)^2}} \quad \mathrm{ \rm and } \quad r''(t)=\frac{f''(s)}{(1+f'(s)^2)^2},
\end{equation}
for any $t=h(s)$. Thus, $r$ satisfies (\ref{admissible}) and it is an admissible function.  Clearly, we have (\ref{hyperquadric})
where $Q(r)=\psi (I\times _{f}\S^{n})$ and $\psi:I\times _{f}\S^{n}\to \L^{n+2}$ is the isometric embedding (\ref{isometric_embedding}) constructed in \cite{Akbar}. 

\vspace{1mm}

Conversely, if $Q(r)$ is defined by (\ref{hyperquadric}) for an admissible function $r$, then we consider $\widetilde{h} : J \to \R$ defined by $\widetilde{h}'(t)=\sqrt{1-r'(t)^2}>0$, $\widetilde{h}(0)=0$, for all $t\in J$. Thus, $\widetilde{h}$ is strictly increasing, and, therefore, exists an open interval $I$ of $\R$, with $0\in I$, such that $\widetilde{h} : J \to I$ is a diffeomorphism. Now define $f : I \to \R$ by $f(s)=(r\circ \widetilde{h}^{-1})(s)>0$, for all $s\in I$ and let $I\times _{f}\S^{n}$ be the corresponding spherical RW spacetime. If we put $h=\widetilde{h}^{-1} : I \to J$, then $h(0)=0$ and $h'(s)=1/\sqrt{1-r'(h(s))^2}=\sqrt{1+f'(s)^2}$, for all $s\in J$, thanks to (\ref{RWEQ}). Finally, we have $Q(r)=\psi (I\times _{f}\S^{n})$.

\begin{remark}\label{251122B}{\rm
Observe that if in the previous construction we replace $h'(s)=\sqrt{1+f'(s)^2}$ with $h'(s)=-\sqrt{1+f'(s)^2}$, then, the corresponding isometric embedding is congruent to the given in (\ref{isometric_embedding}).  
}	
\end{remark}

\begin{remark}\label{251122}
{\rm
For $r(t)=\sqrt{R^2+t^2}$, $t\in\R$, $R>0$ constant, the corresponding hypersurface $Q(r)$ is the $(n+1)$-dimensional De Sitter spacetime $\S^{n+1}_1(R)$ of radius $R$, hence of constant sectional curvature $1/R^2$. On the other hand, if $r=R>0$ then $Q(r)$ is the spherical static Einstein spacetime $\R\times \S^n(R)$. }
\end{remark}

The tangent space of $Q(r)$ at $(t,x)$ is given by
\begin{equation}\label{tangent_space}
T_{(t,x)}Q(r)=\Big\{(a,v)\in \R \times \R^{n+1}:-r(t)r'(t)a+\sum_{i=1}^{n+1}v_{i}x_{i}=0\Big\}=(r(t)r'(t),x)^{\perp}
\end{equation}
where $\perp$ denotes the orthogonal subspace in $\L^{n+2}$ of  the spacelike vector $(r(t)r'(t),x)$ for every $(t,x)\in Q(r)$. 

\vspace{1mm}

As a Lorentzian manifold, $Q(r)$ is time orientable. Indeed, the tangent vector field $T\in \mathfrak{X}(Q(r))$ given by
\begin{equation}\label{tangent}
T(t,x)=\frac{1}{\sqrt{1-r'(t)^2}}\Big(1,\frac{r'(t)}{r(t)}x\Big)
\end{equation} 
for every $(t,x)\in Q(r)$, satisfies $\langle T,T\rangle =-1$ everywhere on $Q(r)$. Precisely, $T$ is the normalization of the timelike conformal symmetry 
$K=\psi_{*}(f\,\partial/\partial t)\in \mathfrak{X}(Q(r))$, that is
\begin{equation}\label{closed_conformal_vector_field}
K(t,x)=\frac{1}{\sqrt{1-r'(t)^2}}\Big(r(t),r'(t)x\Big),
\end{equation}
for all $(t,x)\in Q(r)$. Therefore, it satisfies
\begin{equation}\label{closed_conformal_vector_field_2}
\overline{\nabla}_VK=\frac{r'(t)}{\sqrt{1-r'(t)^2}}\,V
\end{equation}
for any $V\in T_{(t,x)}Q(r)$. Thus, the vector field $K$ on $Q(r)$ is conformal with $\mathcal{L}_K\langle\;,\;\rangle =2\rho\,\langle\;,\;\rangle$, where $\rho (t,x)=r'(t)/\sqrt{1-r'(t)^2}$ and the metrically equivalent $1$-form to $K$ is closed.   

\vspace{1mm}

If for each $(t,x)\in Q(r)$ we set $\mathcal{D}_r(t,x)=K(t,x)^{\perp}$ then $\mathcal{D}_r$ is a distribution on $Q(r)$. Note that (\ref{closed_conformal_vector_field_2}) gives that $\mathcal{D}_r$ is integrable and each leaf $t=t_0$ is totally umbilical in $Q(r)$ with constant mean curvature. Moreover, $t=t_0$ is totally geodesic if and only if $r'(t_0)=0$. 

\vspace{1mm}

Obviously, $Q(r)$ has the same intrinsic geometry as that of the spherical RW spacetime from which it came.

\vspace{1mm}

\begin{remark}\label{relacion_con_MS}
{\rm A Lorentzian manifold $(M,g)$ admits an isometric embedding in an $N$-dimensional Lorentz-Minkowski spacetime $\mathbb{L}^N$ if and only if $(M,g)$ is a stably causal spacetime \cite[p. 63]{BEE} and admits $\tau\in C^{\infty}(M)$ such that $g(\nabla \tau,\nabla \tau)\leq -1$ \cite[Thm. 1.1]{MS}. Clearly, the function $\tau : I\times _{f}\S^{n}\rightarrow \mathbb{R}$, given by $\tau(t,p)=t$ is smooth and its gradient satisfies $\nabla \tau =-\partial/\partial t$. Therefore, the spacetime $I\times _{f}\S^{n}$ is stably causal \cite{BS}. Moreover, $g(\nabla \tau,\nabla \tau)=-1$ everywhere, thus $I\times _{f}\S^{n}$ lies under the assumptions of  \cite[Thm. 1.1]{MS} and hence, it is isometrically embeddable in $\mathbb{L}^N$, indeed, formula (\ref{isometric_embedding}) asserts that, in this case, $N=n+2$.
}
\end{remark} 

\begin{remark}\label{comparacion_con_caso_riemanniano}
{\rm Given $f\in C^{\infty}(I)$, $f>0$, consider now on $I\times \mathbb{S}^n$ the Riemannian metric $g_{f}:=\pi_{I}^{*}(ds^{2})+f(\pi)^{2}\pi^{*}(g),$ (compare with (\ref{g^f})). Thus, we have a Riemannian warped product $(I\times \mathbb{S}^n,g_f)$. The analogous construction to (\ref{isometric_embedding}) defines an isometric embedding from $(I\times \mathbb{R}^n,g_f)$ to $\R^{n+2}$ if and only if $h'(s)^2+f'(s)^2=1$, for all $s\in I$. 
Therefore, assume $|f'|<1$ and let $h\in C^{\infty}(I)$ given by $h'(s)=\sqrt{1-f'(s)^2}>0$ for all $s\in I$ and $h(0)=0$. Take as above $J:=h(I)$ and $r:=f\circ h^{-1}>0$. In this case, $r'(t)=f'(s)/\sqrt{1-f'(t)^2}$, where $h(s)=t$ (compare with (\ref{RWEQ})). Hence, the condition $|r'|<1$ does not hold here as in the Lorentzian case. Now, the map $\varphi : I \times \mathbb{S}^n \rightarrow \mathbb{R}^{n+2}$ defined by $\varphi(s,p)=(h(s), f(s)p)$, is an isometric embedding of $(I\times \mathbb{R}^n,g_f)$ in $\mathbb{R}^{n+2}$. Moreover $\varphi(I\times \mathbb{S}^n)=\{(t,x)\in J\times \mathbb{R}^{n+1}\, :\,\|x\|^2=r(t)^2\}$ is a rotation hypersurface in $\mathbb{R}^{n+2}$, \cite{doCD}.

}
\end{remark}

\section{Set up}

Now, from a extrinsic point of view, each $Q(r)$ is a rotation hypersurface of $\L^{n+2}$ in the terminology of \cite{doCD}. In fact, for a given admissible function $r\in C^{\infty}(J)$, let us consider the curve $\gamma : J \to \L^{n+2}$, given by $\gamma(t)=(t,r(t),\overbrace{0,..., 0}^{(n)})$. Note that the first assumption on $r$ in (\ref{admissible}) implies that the image of $\gamma$ does not meet the timelike axis $x_j=0$, $j=1,...,n+1$. On the other hand, the second one means that $\gamma$ is timelike. Denote now by $O^1(n+2)$ the group of linear isometries of $\L^{n+2}$ and by $G$ the subgroup $\{1\}\times O(n+1)$ of $O^1(n+2)$. Then, we have
\begin{equation}\label{rotation}
Q(r)=\{\,(A\circ \gamma)(t)\; :\; A\in G,\, t\in J\, \}\,.
\end{equation}
Thus, $Q(r)$ is a rotation hypersurface of $\L^{n+2}$ with profile curve $\gamma$ and rotation axis $x_j=0$, $j=1,...,n+1$. Note that if $E_t$ is the orthogonal hyperplane to the rotation axis through the point $(t,\overbrace{0,...,0}^{n+1})$, then $E_t$ is spacelike, and therefore, identifiable to the Euclidean space $\R^{n+1}$. Observe also that the slice $Q(r)\cap E_t$ is an $n$-dimensional round sphere in $E_t$ with radius $r(t)$.

\vspace{1mm}

From (\ref{tangent_space}) we have that the Lorentzian hypersurface $Q(r)$ of $\L^{n+2}$ admits a unit spacelike normal vector field $N\in \mathfrak{X}^{\perp}(Q(r))$, given  at $(t,x)\in Q(r)$ by 
\begin{equation}\label{normal}
N(t,x)=\frac{1}{r(t)\sqrt{1-r'(t)^2}}\Big(r'(t)r(t),x\Big)\,.
\end{equation}

\vspace{1mm}

Now, let us denote by $\nabla^0$ and $\overline\nabla$ the Levi-Civita connections of $\L^{n+2}$ and $Q(r)$, respectively.
For any $V,W\in\mathfrak{X}(Q(r))$, the Gauss and Weingarten formulae of $Q(r)\subset \L^{n+2}$ are respectively written as
\begin{equation}\label{Gauss_formula}
\nabla^0_VW=\overline\nabla_VW+\langle AV,W\rangle N,
\end{equation}
\begin{equation}\label{Weingarten_formula}
\hspace*{-21mm}\nabla^0_VN=-A(V),
\end{equation}
where $A$ is the Weingarten operator with respect to $N$, that is explicitly given by the following result.

\begin{lemma}\label{quasiumbilical}
The Weingarten operator $A$ with respect to $N$ of the Lorentzian hypersurface $Q(r)\subset \L^{n+2}$ is given by
\begin{equation}\label{Weingarten} 
 A(V)=\alpha(t)\, V+\beta(t)\, \big\langle T(t,x), V\big\rangle\, T(t,x)
 \end{equation}
 for all $V\in T_{(t,x)}Q(r)$, where
 \begin{equation}\label{coefficients}
\alpha(t)=\frac{-1}{r(t)\sqrt{1-r'(t)^2}}, \quad  \beta(t)=\frac{r''(t) r(t)+r'(t)^2 -1}{r(t)(1-r'(t)^2)^{3/2}}\,.
\end{equation}
\end{lemma}
\begin{proof}
Write $V=(a,v)\in T_{(t_0,x_0)}Q(r)$,  i.e., $a\,r'(t_0)r(t_0)=\langle x_{0}, v\rangle$, and consider a curve $s \mapsto (t(s),x(s))$ in $Q(r)$ such that $(t(0),x(0))=(t_0,x_0)$ and $(t'(0),x'(0))=(a,v)$. 

\vspace{1mm}

Using (\ref{normal}), we have,
$$
A(V)=-\frac{d}{ds}\Big{|}_{s=0}N(t(s), x(s))=
$$
\begin{equation}\label{100222A}
\hspace*{18mm}-a\frac{r'(t_{0})}{r(t_{0})}\beta(t_{0})\Big(r'(t_{0})r(t_{0}),x_{0}\Big)
\end{equation}
$$
\hspace*{27mm}+\alpha(t_{0})\Big( a[r''(t_{0})r(t_{0})+ r'(t_{0})^2], v\Big).
$$
First of all, suppose that $\langle T(t_0,x_0), V\rangle=0$. Then, from $ar(t_0)=r'(t_0)\langle x_{0}, v\rangle=ar'(t_0)^{2}r(t_0)$ we obtain  $a=0$. Hence, equation (\ref{100222A})
reduces to
$$
A(V)=\alpha(t_{0})V.
$$
Now, for the case $V=T(t_{0},x_{0})$, we have from (\ref{tangent}) and (\ref{100222A}),
$$
\hspace*{-28mm}A(T(t_{0},x_{0}))=
\frac{-r'(t_{0})}{r(t_{0})\sqrt{1-r'(t_{0})^2}}\beta(t_{0})\Big(r'(t_{0})r(t_{0}),x_{0}\Big)
$$
$$
\hspace*{2mm}-\beta(t_{0})\sqrt{1-r'(t_{0})^2}\Big(1, 0\Big)+\alpha(t_{0})T(t_{0},x_{0})=
$$
$$
\hspace*{34mm}=-\beta(t_{0})\Big(\frac{r'(t_{0})^2}{\sqrt{1-r'(t_{0})^2}}+\sqrt{1-r'(t_{0})^2},\frac{r'(t_{0})}{r(t_{0})\sqrt{1-r'(t_{0})^2}}x_{0}\Big)
$$
$$
+\alpha(t_{0})T(t_{0},x_{0})=(\alpha(t_{0})-\beta(t_{0}))T(t_{0},x_{0})\,,
$$
which ends the proof.
\end{proof}

\smallskip

\begin{remark}\label{2 de mayo}
{\rm
Now we are in position to analyze the extrinsic geometry of $Q(r)$. First, let us observe that, obviously, $Q(r)$ is not totally geodesic for any admissible function $r$. On the other hand, Lemma \ref{quasiumbilical} says that $Q(r)$ is quasiumbilical \cite{Chen}. Observe that thanks to (\ref{Weingarten}) we have 
$$
\begin{cases}
A(T)=\big(\alpha(t)-\beta(t)\big)T, and\; \\
A(V)=\alpha(t)V,\; \text{for any}\; V\in T_{(t,x)}Q(r)\; \text{orthogonal to} \;T(t,x). 
\end{cases}
$$
%$$  
%A(T)=\big(\alpha(t)-\beta(t)\big)T, and\;
%A(V)=\alpha(t)V,\; \text{for any}\; V\in T_{(t,x)}Q(r)\; \text{orthogonal to} \;T(t,x). 
%$$ 
Note that $A$ is diagonalizable at any point (recall that a self-adjoint operator respect to a Lorentzian scalar product is not necessarily diagonalizable).
Therefore, $Q(r)$ is totally umbilical if and only if $\beta(t)=0$ everywhere on $Q(r)$, i.e., if and only if the function $r$ satisfies the differential equation $r(t)r''(t)+ r'(t)^2=1$ or equivalently $\frac{d^2}{dt^2}\,r(t)^2=2$. Consequently,
$r(t)^{2}=t^{2}+at+b$,
where $b=r(0)^2>0$ since $r>0$, and $a=2r(0)r'(0)$ with $a^2 <4b$, making use of $|r'|<1$.
Summarizing, we have obtained that $Q(r)$ is totally umbilical if and only if
\begin{equation}\label{totally_umbilical}
r(t)=\sqrt{t^2+at+b}\,,\; \text{where the constants satisfy}\; b>0\; \text{and}\; a^2<4b.  
\end{equation}
}
\end{remark}

\begin{remark}\label{DeSitterQr}
{\rm If $Q(r)$ is totally umbilical then $\alpha(t,x)=\frac{-2}{\sqrt{4b-a^2}}$, and, therefore, the Weingarten operator is $A=\frac{-2}{\sqrt{4b-a^2}}\,I$, where $I$ denotes the identity operator, everywhere.  In this case,  the Gauss equation of $Q(r)$ in $\L^{n+2}$ gives for the curvature tensor $R$ of $Q(r)$, the following expression $R(X,Y)Z=\big(\frac{4}{4b-a^2}\big)\{\langle Y,Z\rangle X - \langle X,Z\rangle Y\}$, i.e., $Q(r)$ has sectional curvature $\frac{4}{4b-a^2}$, indeed, $Q(r)$ is, up to a translation, the $(n+1)$-dimensional De Sitter spacetime $\S^{n+1}_1(\frac{\sqrt{4b-a^2}}{2})$.  }
\end{remark}

From Lemma \ref{quasiumbilical}, the mean curvature function of $Q(r)$ with respect to $N$, $H:=\frac{1}{n+1}\,\text{trace}(A)$, satisfies
\begin{equation}\label{mean_curvature}
H(t,x)=\alpha(t)-\frac{\beta(t)}{n+1}=-\frac{n(1-r'(t)^2)+r''(t)r(t)}{(n+1)r(t)\big(1- r'(t)^2\big)^{3/2}}\,.
\end{equation}

\begin{remark}\label{relation_with_Rafael_Lopez}{\rm In the case $n=1$, formula (\ref{mean_curvature}) agrees, up the sign of $H$ due to our choice of $N$, with \cite[eq. (8)]{Rafael_Lopez}. Therefore, if $H$ is constant then 
$$
\frac{r(t)}{\sqrt{1-r'(t)^2}}+r(t)^2H=\text{constant}\,,
$$
for any $t\in J$, i.e., we have \cite[eq. (9)]{Rafael_Lopez}. However, for $n>1$ no extension of this fact holds, and we have to make a different reasoning.}

\end{remark}
Now, taking into account
$$
\alpha'(t)=-\frac{r'(t)}{r(t)}\,\beta(t), 
$$
we get from (\ref{mean_curvature})

\begin{corollary}
The Lorentzian hypersurface $Q(r)$ of $\L^{n+2}$ has constant mean curvature if and only if there exists $c\in \R$ such that
\begin{equation}\label{constant_mean_curvature}
\beta(t)=\frac{c}{r^{n+1}(t)}\,, 
\end{equation}
equivalently
\begin{equation}\label{constant_mean_curvature_equiv}
\frac{r(t)^n\big(r''(t) r(t)+r'(t)^2 -1\big)}{(1-r'(t)^2)^{3/2}}=c 
\end{equation}
 for all $t\in J$. In this case, we have
 \begin{equation}\label{relation}
 \alpha(t)=\frac{c}{(n+1)r^{n+1}(t)}+H\,. 
 \end{equation}
\end{corollary}

\begin{remark}
{\rm Clearly, for the choice $c=0$ we have the totally umbilical case. On the other hand, for each $c<0$  we have that $r=\sqrt[n]{-c}$ is a solution of (\ref{constant_mean_curvature_equiv}) giving $Q(r)=\R\times \S^{n}(r)$ which has constant mean curvature but it is not totally umbilical.}
\end{remark}
\smallskip
 Extending \cite{CH} we are going to explore which Lorentzian hypersurfaces of the family $Q(r)$ 
have the property that the principal curvature in the axial direction is a constant multiple of the common value of the principal curvatures in the rotational directions.  

\vspace{1mm}

The Lorentzian hypersurface $Q(r)$ of $\mathbb{L}^{n+2}$ has proportional principal curvatures when there exists $\lambda\in \R$ such that
\begin{equation}\label{ppc_def}
\lambda\, \alpha(t)=\alpha(t)-\beta(t),
\end{equation}
for all $t\in J$.
If $Q(r)$ has proportional principal curvatures, then it is totally umbilical when $\lambda=1$ and, using (\ref{mean_curvature}), it has zero mean curvature when $\lambda=-n$.

\vspace{1mm}

As a direct consequence of formulae (\ref{coefficients}) we have,
\begin{proposition}\label{ppc}
The Lorentzian hypersurface $Q(r)$ of $\mathbb{L}^{n+2}$ has proportional principal curvatures, i.e., it satisfies $(\ref{ppc_def})$ if and only if
\begin{equation}\label{190422A}
r(t)r''(t)=\lambda (1- r'(t)^{2})
\end{equation}
for all $t\in J$.
\end{proposition}

\begin{remark}
{\rm An analogous family of differential equations to (\ref{190422A}) appears in \cite{CH} in the study of  hypersurfaces of revolution with proportional principal curvatures in Euclidean spaces, giving the corresponding  solutions in  \cite[Thm. 1]{CH}.
In our setting,  the solutions of (\ref{190422A}) for some choices of $\lambda \in \R$ do not provide Lorentzian hypersurfaces of $\L^{n+2}$. 
Indeed for the choice $\lambda=-1$ the set of such solutions of (\ref{190422A}) is given by 
$$
r(t)=\frac{1}{b}\sinh(bt+c)
$$
where $b\neq 0$, $c\in \R$ and $t>-c/b$. Since the condition $r'(t)^2<1$, for all $t\in J$ does not hold,  the corresponding hypersurface does not inherit 
a Lorentzian metric from $\L^{n+2}$.

}
\end{remark}

\section{Stationary spacelike submanifolds} 
Let $\Psi : M^k \to I\times_f \S^n$ be a spacelike immersion of a spherical RW spacetime  $I\times_f \S^n$. After identifying $\Psi$ to $\psi \circ \Psi$, for a suitable admissible function $r$, we can consider 
\begin{equation}\label{Psi}
\Psi : M^k \to Q(r)\subset  \L^{n+2}.
\end{equation}
Now, let  $\sigma$ and $\tilde \sigma$ be the second fundamental forms of  $\Psi : M^k \to \L^{n+2}$ and $\Psi : M^k \to Q(r)$, respectively. From (\ref{Gauss_formula}) we have
\begin{equation}\label{second_fundamental_form}
\sigma(X,Y)=\tilde \sigma (X,Y) \,+\, \langle AX,Y \rangle N\,,
\end{equation}
for all $X,Y\in \mathfrak{X}(M^k)$, where $N$, given in  (\ref{normal}), is the unit normal vector field of $Q(r)$ in $\L^{n+2}$ and $A$ its corresponding Weingarten operator (\ref{Weingarten}). Consequently, the respective mean curvature vector fields 
$\mathbf{H}$ and $\widetilde{\mathbf{H}}$ are related by
\smallskip

\begin{equation}\label{mean_curvatures_relation}
\mathbf{H}=\widetilde{\mathbf{H}}\, +\,\frac{1}{k}\,\sum_{i=1}^k\,\langle AX_i,X_i \rangle N\,,
\end{equation}
where $\{X_i\}$ is a local orthonormal basis of tangent vector fields to $M^k$. 
Then, a direct computation from Lemma \ref{quasiumbilical} gives
that $\widetilde{\mathbf{H}}=0$ if and only if 
\begin{equation}
\mathbf{H}=\Big(\alpha(\Psi_0 )\,+\, \frac{\beta(\Psi_0 )}{k}\, \|T^{\top} \|^{2}\Big) \,N,
\end{equation}
where $T^{\top}$ denotes the tangential component of $T$ on $M^k$.

\vspace{1mm}

By means of the Beltrami equation (\ref{Beltrami}), we have then
\begin{equation}
\Delta \Psi=\Big(k\,\alpha(\Psi_0)\,+\, \beta(\Psi_0 )\, \|T^{\top} \|^{2}\Big)\, N\,.
\end{equation}
Now, for any local orthonormal tangent frame $\{X_1,...,X_k\}$, we compute
\begin{align*}
\|T^{\top} \|^{2} &=\frac{1}{1- r'(\Psi_{0})^{2}}\sum_{i=1}^{k}\Big\langle \Big(1, \frac{r'(\Psi_{0})}{r(\Psi_{0})}\Psi_{1}, \dots ,\frac{r'(\Psi_{0})}{r(\Psi_{0})}\Psi_{n+1}\Big), \Big(X_{i}(\Psi_{0}),\dots , X_{i}(\Psi_{n+1})\Big) \Big\rangle^{2}\\[2mm]
&=\frac{1}{1- r'(\Psi_{0})^{2}}\sum_{i=1}^{k}\Big(-X_{i}(\Psi_{0})+  \frac{r'(\Psi_{0})}{r(\Psi_{0})}\Psi_{1}X_{i}(\Psi_{1})+\dots +\frac{r'(\Psi_{0})}{r(\Psi_{0})}\Psi_{n+1}X_{i}(\Psi_{n+1}) \Big)^{2}\\[2mm]
&=\frac{1}{1- r'(\Psi_{0})^{2}}\sum_{i=1}^{k}\Big(-X_{i}(\Psi_{0})+  r'(\Psi_{0})^{2}X_{i}(\Psi_{0})\Big)^{2}\\[2mm]
&=(1- r'(\Psi_{0})^{2})\sum_{i=1}^{k}(X_{i}(\Psi_{0}))^{2}\\[2mm]
&=(1- r'(\Psi_{0})^{2})\| \nabla \Psi_{0} \|^{2}.
\end{align*}
Therefore, we get
$$
k\,\alpha(\Psi_0 )+ \beta(\Psi_0 )\, \|T^{\top} \|^{2}=\alpha(\Psi_0 )\,\Big[k-\Big(r''(\Psi_{0})r(\Psi_{0})+  r'(\Psi_{0})^{2}-1\Big)\| \nabla \Psi_{0} \|^{2}\Big]\,.
$$

We summarize previous computations as follows,
\begin{proposition}\label{lemma11}
A spacelike immersion $ \Psi\colon M^k \to Q(r)$ is stationary if and only if
\begin{equation}\label{lemma11_formula}
\Delta \Psi=\alpha(\Psi_0 )\,\Big[k-\Big(r''(\Psi_{0})r(\Psi_{0})+  r'(\Psi_{0})^{2}-1\Big)\| \nabla \Psi_{0} \|^{2}\Big]\,N.
\end{equation}
Equivalently, $\Psi$ is stationary if and only if 
$$
\Delta \Psi +q_{\Psi_{0}}\,\mathbf{P}=0,
$$
as announced in $(\ref{120322A})$, where  $\mathbf{P}$ is the vector field along the immersion $\Psi$ given in $(\ref{110322A})$
and 
$q_{\Psi_{0}}$ the function on $M^k$ given in $(\ref{110322B})$.

\end{proposition}

\begin{remark}\label{sabado}
{\rm The previous result gives rise to the following system of elliptic partial differential equations for the components of the spacelike immersion $\Psi\colon M^k \to \L^{n+1}$\,,
$$
\begin{cases} \Delta \Psi_{0}+q_{\Psi_{0}}\,r(\Psi_{0})\,r'(\Psi_{0})=0, \\[2.5mm] \Delta \Psi_{i}+q_{\Psi_{0}}\, \Psi_{i}=0, \quad \quad i=1,2,\dots, n+1.
\end{cases}
$$
}
\end{remark}

Assume $r(t)=\sqrt{1+t^2}$ for all $t\in \R$. In this case $Q(r)$ is isometric to the unitary De Sitter spacetime  $ \S^{n+1}_{1}$. A spacelike immersion  $\Psi \colon M^{k}\rightarrow Q(r)\subset \L^{n+2}$ is stationary if and only if
\begin{equation}\label{case De Sitter}
\Delta \Psi+k \, \Psi=0.
\end{equation}

 On the other hand, if $r(t)=1$ for all $t\in \R$, then $Q(1)$ is isometric to the static Einstein spacetime   $\R\times \S^{n}$. Hence, as a direct consequence of Proposition \ref{lemma11}, we obtain,
 \begin{corollary}\label{caso_particular}
 A spacelike immersion $\Psi\colon M^{k}\rightarrow Q(1)\subset \L^{n+2}$ is stationary if and only if
\begin{equation}\label{case Q(1)}
\Delta \Psi_{0}=0, \quad 
\Delta \Psi_{i}+(k+ \|\nabla \Psi_{0} \|^{2})\Psi_{i}=0,\quad i=1,2,\dots, n+1.
\end{equation}
\end{corollary}
\noindent As an immediate consequence, the only compact stationary spacelike submanifolds in $Q(1)$ are the stationary submanifolds of a slice $\{t_0\}\times \S^{n}\equiv \S^{n}$.
\vspace{0.15mm}

\begin{proposition}
Given a stationary spacelike immersion $\Psi\colon M^k \to Q(r)$, if the function 
$r(\Psi_{0})^2$ on $M^k$ attains a local maximum value at $x\in M^k$ then $q_{\Psi_{0}}(x)>0$. 
\end{proposition}
\begin{proof}
In fact, this follows from $\sum_{j=1}^{n+1}\Psi_j^2=r(\Psi_{0})^2$ that gives 
$$
\frac{1}{2}\,\Delta r(\Psi_{0})^2+q_{\Psi_{0}}r(\Psi_{0})^2=\sum_{j=1}^{n+1}\|\nabla \Psi_j\|^2>0,
$$
using Remark \ref{sabado}.
\end{proof}

\begin{corollary}\label{130322A}
Let $\Psi\colon M^k \to Q(r)$ be a stationary  compact spacelike immersion. Assume $q_{\Psi_{0}}>0$ and $r'\leq 0$ or $r'\geq 0$. Then, $\Psi\colon M^k \to Q(r)$ factors through a slice $\Psi_{0}=\mathrm{cte}$ with $r'(\Psi_{0})=0.$ Therefore, $q_{\psi_0}$ is  a positive constant and  $\Psi\colon M^k \to Q(r)$ realizes a stationary immersion in a totally geodesic slice of $Q(r)$, which is isometric to an $n$-dimensional round sphere of radius $r(\Psi_{0})$.  In particular, there is no compact stationary spacelike submanifold in $Q(r)$ which is contained in a slab where $r'<0$ or $r'>0.$
\end{corollary}
\begin{proof}
From Remark \ref{sabado},  we have
$$
\int_{M^k}q_{\Psi_{0}}r(\Psi_{0})r'(\Psi_{0})d\mu_{g}=0,
$$
where $d\mu_{g}$ denotes the canonical measure associated to the induced metric.
Therefore,  we get $r'(\Psi_{0})=0$ and  Remark \ref{sabado} implies that $\Delta \Psi_{0}=0$. The compactness of $M^k$ shows that $\Psi_{0}=\mathrm{cte}$, and $r'(\Psi_{0})=0$ implies that the corresponding slice is totally geodesic in $Q(r)$. Now, equation (\ref{120322A}) reduces to
$$\Delta \Psi_{i}+\frac{k}{r(\Psi_{0})^2}\Psi_{i}=0,\quad i=1,2,\dots, n+1,$$
and now we call the aforementioned Takahashi result \cite[Thm. 3]{Takahashi} to end the proof.
\end{proof}

In order to provide a physical interpretation to the assumptions in Corollary \ref{130322A}, let us
recall that for a given reference frame $U$ in a spacetime $\overline M$ in terminology of \cite[Def. 2.3.1]{SW}, the observers in $U$ are {\it spreading out} (resp.{\it coming together}) if div$(U)>0$ (resp. div$(U)<0$) \cite[p. 58]{SW}. Thus, for the observers in $\overline M$ their universe is {\it expanding} (resp. {\it contracting}). In the case $\overline M = I\times_{f} \mathbb{S}^{n}$ and $U=\partial_t$, the co-moving reference frame, we have div$(\partial_t)=n\,f'/f$. Therefore, the spherical RW spacetime $I\times_{f} \mathbb{S}^{n}$ is expanding (resp. contracting) (for co-moving observers) if $f'(t)>0$ (resp. $f'(t)<0$) for all $t\in I$.  

\vspace{1mm}

On the other hand, a spacetime $\overline M$ obeys the {\it Null Convergent Condition} if its Ricci tensor satisfies $\overline{\mathrm{Ric}}(X,X)\geq 0$ for all null tangent vector $X$. This assumption is a necessary mathematical condition that holds from the physical fact that gravity attracts on average. Moreover, it also holds that  if the spacetime obeys the Einstein equation (with zero cosmological constant) for suitable stress-energy tensors. In the case $\overline M = I\times_{f} \mathbb{S}^{n}$,  the 
Null Convergence Condition holds if and only if 
\begin{equation}\label{NCC}
f^2(\log f)''\leq 1,
\end{equation}
(see \cite{ARS}, for instance).  

\vspace{1mm}

Now, take into account, one more time, that the spherical Robertson-Walker spacetime $I\times_{f}\S^{n}$  is isometric to $Q(r)$ by means of (\ref{isometric_embedding}). Consequently, from equations in \eqref{RWEQ}, we get
$$
r''(t)r(t)+r'(t)^2-1=\frac{f''(s)f(s)-f'(s)^2-1}{(1+f'(s)^2)^2},
$$
for $h(s)=t$, and the Null Convergence Condition implies that $q_{\Psi_{0}}\geq 0$ holds for every stationary spacelike immersion in $Q(r)$. Moreover, we would like to point out that   $\overline M = I\times_{f} \mathbb{S}^{n}$ is Einstein if and only if
$$
r''(t)r(t)+r'(t)^2-1=\frac{f''(s)f(s)-f'(s)^2-1}{(1+f'(s)^2)^2}=0.
$$
(see \cite{ARS1997}). But, $I\times_f \S^n$ is Einstein if and only if it has (positive) sectional curvature \cite[Table]{ARS1997}. Note that $Q(r)$ must be totally umbilical in $\L ^{n+2}$ (see Remark \ref{2 de mayo}).

\vspace{1mm}

Summarizing, when $Q(r)$ satisfies the Null Convergence Condition, then for every $k$-dimensional stationary spacelike immersion $\Psi$ in $Q(r)$, we have,
$$
q_{\Psi_{0}}\geq k\,\alpha(\Psi_{0})^2
$$
with equality whenever $Q(r)$ has (positive) constant sectional curvature (hence, $Q(r)$ is an open portion of a De Sitter spacetime).

\vspace{1mm}

As a consequence of previous discussion, we have,

\begin{corollary}\label{3105B} There is no stationary compact spacelike submanifold in a expanding or contracting spherical RW spacetime $I\times_{f} \mathbb{S}^{n}$ satisfying the Null Convergence Condition. 
\end{corollary}
\begin{proof}Let us argue by contradiction. Suppose there exists  a stationary compact spacelike immersion satisfying in a such t spherical RW spacetime $I\times_{f} \mathbb{S}^{n}$. Now Corollary \ref{130322A} can be applied since  $f'$ is signed and $f''(s)f(s)- f'(s)^2\leq 1$, therefore, we should have $f'(\Psi_0)=0$, which is a contradiction. 
\end{proof}

\smallskip
We finish the paper with the following result that completes the main result as announced in the end of introduction.

\begin{theorem}\label{010222}
Let  $r$ be an admisible function and $\Psi\colon M^k \to \L^{n+2}$  any spacelike immersion with $q_{\Psi_{0}}>0$.
 If $\Psi$ satisfies $(\ref{120322A})$, then $\Psi$ realizes a stationary spacelike immersion in $Q(r)$.
\end{theorem}
\begin{proof}  
Taking into account  Proposition \ref{lemma11}, we only need to  show that $\Psi(M^{k})\subset Q(r)$.
In fact, from (\ref{120322A}) and (\ref{Beltrami}), we have
\begin{equation}\label{110322B1}
k\mathbf{H}+ q_{\Psi_{0}}\mathbf{P}=0,
\end{equation}
that implies that vector field $\mathbf{P}$  along the spacelike immersion $\Psi\colon M^k \to \L^{n+2}$ (\ref{110322A}) is normal  everywhere. 

\vspace{1mm}

\noindent Now, the Weingarten formula  for $\Psi\colon M^k \to \L^{n+2}$ and the normal vector field $\mathbf{P}$  gives
\begin{equation}\label{060822Aa}
\nabla^{0}_{v}\mathbf{P}=
-A_{\mathbf{P}}(v)+ \nabla^{\perp}_{v}\mathbf{P}.
\end{equation}
We compute the left hand side of (\ref{060822Aa}) for every $v=(v_0,v_1,...,v_{n+1})\in T_{x}M^k$, obtaining
$$
\nabla^{0}_{v}\mathbf{P}=\Big(\,\Big[r'(\Psi_{0}(x))^2   +r(\Psi_{0}(x))r''(\Psi_{0}(x))\Big]v_{0}, v_{1}, \dots , v_{n+1}\Big)
$$
\begin{equation}\label{060822A}
\hspace*{13mm}=v+v_{0}\Big[r'(\Psi_{0}(x))^2   +r(\Psi_{0}(x))r''(\Psi_{0}(x))-1\Big]\,\frac{\partial}{\partial t}{\Big|_{\Psi(x)}}.
\end{equation}
Now, recall that for every $a\in \L^{n+2}$, the vector field $a^{\top}=\nabla \langle \Psi, a\rangle \in \mathfrak{X}(M^k)$. Here the superscript $\top$ denotes the tangent part of $a\in \L^{n+2}$ along the immersion $\Psi\colon M^k \to \L^{n+2}$.

\vspace{1mm}

In particular, we get $$\Big(\frac{\partial}{\partial t}{\Big|_{\Psi(x)}}
\Big)^{\top}=\nabla \langle \Psi, e_{0}\rangle=- \nabla\Psi_{0},$$
and therefore, from equations (\ref{060822A}) and  (\ref{060822Aa}) we have
\begin{equation}\label{Tomas1}
A_{\mathbf{P}}=-\mathrm{Id}+\Big[r'(\Psi_{0})^2   +r(\Psi_{0})r''(\Psi_{0})-1\Big]\,d\Psi_{0}\otimes\nabla \Psi_{0}\,.
\end{equation}
Hence,  we can directly compute
\begin{equation}\label{Tomas4}
\mathrm{trace}(A_{\mathbf{P}})=-k+\Big[r'(\Psi_{0})^2   +r(\Psi_{0})r''(\Psi_{0})-1\Big]\| \nabla\Psi_{0}\|^2=- \frac{q_{\Psi_{0}}}{\alpha(\Psi_{0})^2}\,,
\end{equation}
and by means of (\ref{110322B1}), we get
\begin{equation}\label{281122A}
\mathrm{trace}(A_{\mathbf{H}})= k \| \mathbf{H}\|^{2}=\frac{q^{2}_{\Psi_{0}}}{k\,\alpha(\Psi_{0})^2}\,.
\end{equation}
On the other hand, formula (\ref{110322B1}) also gives that
\begin{equation}\label{Tomas5}
\| \mathbf{H}\|^{2}=\frac{q_{\psi_{0}}^2}{k^2}\| \mathbf{P}\|^{2}=\frac{q^{2}_{\Psi_{0}}}{k^2}\Big( -r(\Psi_{0})^2 r'(\Psi_{0})^2+ \sum_{j=1}^{n+1}\Psi_{j}^2\Big)\,.
\end{equation}
The above two formulae (\ref{281122A}) and (\ref{Tomas5})  imply 
$$
 -r(\Psi_{0})^2 r'(\Psi_{0})^2+ \sum_{j=1}^{n+1}\Psi_{j}^2=\frac{1}{\alpha(\Psi_{0})^2},
$$
which ends the proof. 
\end{proof}

\vspace{6mm}

\noindent Danilo Ferreira and Eraldo A. Lima Jr\\
Departamento de Matemática,\\ Universidade Federal da Paraíba,\\ 58051-900 João Pessoa, PB, Brazil\\
danilodfs.math@gmail.com\\
eraldo.lima@academico.ufpb.br 

\vspace{3mm}

\noindent Francisco J. Palomo\\
Departamento de Matemática Aplicada\\
Universidad de Málaga, 29071 Málaga, Spain\\
fpalomo@uma.es

\vspace{3mm}

 \noindent Alfonso Romero\\
 Departamento de Geometría y Topología,\\ Universidad de Granada, 18071 Granada, Spain\\
aromero@ugr.es


\begin{thebibliography}{99}
\bibitem{Akbar} M.M. Akbar, Embedding FLRW geometries in peudo-Euclidean and anti-de Sitter spaces, {\it Physical Review D}, {\bf 95} (2017), 064058(1--10).

\bibitem{ARS} J.A. Aledo, R.M. Rubio and J.J. Salamanca, Complete spacelike hypersurfaces in generalized Robertson–Walker and the null convergence condition: Calabi–Bernstein problems, {\it RACSAM,}  {\bf 111} (2017), 115--128.  

\bibitem{ARS1997} L.J. Al{\'i}as, A. Romero and M. S\'anchez, Spacelike hypersurfaces of constant mean curvature and Calabi-Bernstein type problems, 
{\it Tohoku Math. J.}, {\bf 49} (1997), 337--345

\bibitem{BEE} J.K. Beem, P.E. Ehrlich and K.L. Easley, {\it Global Lorentzian Geometry, second edition}, Monographs and Textbooks in Pure and Applied Mathematics, {\bf 202}, Marcel Dekker, 1996.

\bibitem{BS} A.N. Bernal and M. S\'anchez, Smoothness of time functions and the metric splitting of globally hyperbolic spacetimes, {\it Commun. Math. Phys.}, {\bf 257} (2005) 43--50.

\bibitem{doCD} M. do Carmo and M. Dajczer, Rotation hypersurfaces in spaces of constant curvature,  {\it Trans. Amer. Math. Soc.}, {\bf 227} (1983), 685--709.

\bibitem{Chen} B.-Y. Chen, {\it Geometry of Submanifolds}, Marcel Dekker,
New York, 1973.

\bibitem{Chen2} B.-Y. Chen, On the total curvature of immersed manifolds IV: Spectrum and total mean curvature, {\it Bull. Inst.
Math. Acad. Sinica } {\bf 7} (1979) 301--311.

\bibitem{CH} V. Coll and M. Harrison, Hypersurfaces of revolution with proportional principal curvatures, {\it Advances in Geometry}, {\bf 13} (2013), 485--496.

\bibitem{FLR} D. Ferreira, E.A. Lima Jr. and A. Romero, Complete stationary spacelike surfaces in an $n$-dimensional Generalized Robertson-Walker spacetime,
Mediterranean J. Math. (2022), (to appear). 

\bibitem{Rafael_Lopez} R, L\'opez, Timelike surfaces with constant mean curvature in Lorentz three-space, {\it Tohoku Math. J.}, {\bf 52} (2000), 515--532.

\bibitem{Markvorsen} S. Markvorsen, A Characteristic Eigenfunction for Minimal Hypersurfaces in Space Forms, {\it Math Z.}, 
{\bf 202}, (1989), 375--382.  

\bibitem{MS} O. M\"uler and M. S\'anchez, Lorentzian manifolds isometrically embeddable in $\mathbb{L}^N$, {\it Trans. Amer. Math. Soc.}, {\bf 363},  (2011), 5367--5379.

\bibitem{O'Neill} B. O'Neill, {\it Semi-Riemannian Geometry with
Applications to Relativity}, Academic Press, New York, 1983.

\bibitem{SW} R. Sachs and H. Wu, {\it General Relativity for Mathematicians},  Graduate Texts in Math. {\bf 48}, Springer, New York, 1977.

\bibitem{S} M. S\'anchez, On the Geometry of Generalized Robertson Walker Spacetimes: Geodesics, {\it Gen. Relat. Gravitation} , {\bf 30} (1998), 915--932.

\bibitem{Takahashi} T. Takahashi, Minimal immersion of Riemannian manifolds, {\it J. Math. Soc. Japan.}, {\bf 18} (1966), 380--385.
\end{thebibliography}
\end{document}